\documentclass[a4paper,12pt]{article}
\usepackage[cp1251]{inputenc}
\usepackage[ russian, ukrainian, english]{babel}
\usepackage{amsmath, amsthm, amsfonts, amssymb}

\newcommand{\Var}{\mathop{\rm Var}}
\newcommand{\cov}{\mathop{\rm cov}}
\newcommand{\vk}{\varkappa}
\newcommand{\ov}{\overline}

\newcommand{\vf}{\varphi}
\newcommand{\ve}{\varepsilon}
\newcommand{\cT}{{\mathcal T}}

\newcommand{\wt}{\widetilde}
\newcommand{\1}{1\!\!\,{\rm I}}
\newcommand{\mbR}{{\mathbb R}}
\theoremstyle{plain}
\newtheorem{thm}{Theorem}
\newtheorem{lem}[thm]{Lemma}
\theoremstyle{definition}
\newtheorem{defn}[thm]{Definition}
\newtheorem{prop}[thm]{Proposition}
\begin{document}
\large
\begin{center}
{\Large\bf Properties of Gaussian local times }
\end{center}

Andrey Dorogovtsev ${}^a$, Olga Izyumtseva ${}^b$

${}^a$ \ Institute of Mathematics National Academy of Sciences of Ukraine,

01601 Ukraine, Kiev--4, 3, Tereshenkivska st.

${}^b$ \ Institute of Mathematics National Academy of Sciences of Ukraine,

01601 Ukraine, Kiev--4, 3, Tereshenkivska st.
\vskip20pt

{\bf Abstract}

In the paper we investigate properties of local time for one class of Gaussian processes. These processes are called by integrators since every function from $L_2([0;1])$ can be integrated over it. Using the white noise representation one can associate integrators with continuous linear operators in $L_2([0;1]).$ In terms of these operators we discuss existence and properties of local time for integrators. Also, we study the asymptotic behaviour of the Brownian self-intersection local time when its end-point tends to infinity.

{\bf Keywords}: Gaussian integrator, local time, white noise, local nondeterminism, self-intersection local time.

\begin{center}
{\Large\bf Properties of Gaussian local time }
\end{center}

\section{Introduction}
\label{section1}

  The aim of this article is to study the local time for a certain class of Gaussian processes. Since the works of S. Berman \cite{7} the existence and properties of the local time are studied for a wide class of Gaussian processes and fields \cite{17}. In the case of Brownian motion the local time can be investigated using the parabolic equations and potential theory due to the independence of increments and self-similarity. For general Gaussian processes S. Berman proposed the notion of local nondeterminism which in some sense means almost independency of increments on small intervals. Under some technical assumptions this property leads to the existence and regularity of local time with respect to both spatial and time variables. Different authors proposed the version of local nondeterminism property for Gaussian processes and fields and proved not only existence of the local time but investigated some its asymptotic properties such as the law of iterated logarithm or small ball probabilities \cite{18}. Nevertheless the local nondeterminism property is hard to check for an arbitrary Gaussian process. Simple sufficient conditions were given for processes with stationary increments or for self-similar processes \cite{19}. In the case of planar Gaussian process the situation is much worse. Namely, for Brownian motion on the plane the existence of multiple self-intersection points is well-known \cite{20}. The corresponding local time of multiple self-intersection needs to be properly renormalized. Such renormalization was done by S. Varadhan \cite{11}, E.B. Dynkin \cite{12}, J. Rosen \cite{21}. Later J.-F. Le Gall \cite{14} obtained asymptotic expansion for the area of Brownian sausage which contains renormalized self-intersection local times as the coefficients. As all these results are essentially based on the structure of Brownian motion, used here technic can not be expanded on other Gaussian processes. But the question of such generalization is quiet interesting in view of constructing the random polymer models using  not only Markov processes (in some cases there are no reasons for molecule to differ the starting point and the end-point). All mentioned reasons lead to the attempt to find such class of Gaussian processes for which some version of local nondeterminism holds and results related to existence and renormalization of the local time and self--intersection local times can be achieved. Such class of processes was introduced by A.A. Dorogovtsev in \cite{1} in connection with the anticipating stochastic integration.  The original definition is the following.

 \begin{defn}[\cite{1}]
 \label{defn1}
 A centered Gaussian process $x(t),\ t\in[0; 1]$ is said to be an integrator if there exists the constant $c>0$ such that for an arbitrary partition $0=t_0<t_1<\ldots<t_n=1$ and real numbers $a_0, \ldots, a_{n-1} $
 \begin{equation}
 \label{eq1}
 E\Big(\sum^{n-1}_{k=0}a_k(x(t_{k+1})-x(t_k))\Big)^2\leq c\sum^{n-1}_{k=0}a^2_k\Delta t_k.
 \end{equation}
 \end{defn}

 The inequality \eqref{eq1} allows to integrate functions from $L_2([0; 1])$ with respect to $x.$ This naturally leads to the definition of Skorokhod type stochastic integral with respect to $x.$ In \cite{1} the corresponding stochastic calculus including It\^o formula for $x$ was considered. The following statement describes the structure of integrators.

 \begin{prop} [\cite{1}]
 \label{prop2}
 The centered Gaussian process $x(t),\ t\in[0; 1]$ is  an integrator iff there exist a Gaussian white noise $\xi$ \cite{3, 4} in $L_2([0; 1])$ and a continuous linear operator $A$ in the same space such that
 \begin{equation}
\label{eq2}
x(t)=(A\1_{[0; 1]}, \xi),\ t\in[0; 1].
\end{equation}
\end{prop}
In the article we use the language of white noise analysis \cite{3,4,15,16}.
Note that if $A$ equals identity, then $x$ in the expression \eqref{eq2} is a Wiener process. For continuously invertible operator $A$ one can expect that $x$ will inherit some properties of the Wiener process. For example, it will be proved in the section 2 that if $A$ is continuously invertible, then $x$ has a local time at any point $u\in\mbR.$ Such local time can be obtained as the occupation density. Also it will be checked that this density is a continuous function in spatial and time variables.  In the section 3 we will prove a continuous dependence of local times of integrators on generating them operators.
 %vst
 The main method of our investigations is based on the studying of functional properties of the Hilbert-valued functions. In particular, we obtain some estimations for the Gram determinant constructed by the increments of such function.  These estimations allow to investigate the conditional moments of the Brownian self-intersection local time in dimensions one and two when the end point (the value $w(1)$) tends to infinity. We establish the speed of decreasing of the mentioned moments. The question of the conditional behavior of the self-intersection local  times is inspired by the studying properties of continuous polymer  models [22-24].
  The real polymers can not have self-intersections due to the excluded volume effect \cite{22}. But the energy of interaction between the monomers from different places in polymer molecule influences on its form. Flory proposed the evaluation of the size of polymer based on the counting of interaction energy \cite{22}.  The Brownian path can be viewed as an ideal Gaussian  model of polymer \cite{22}. Applying the Flory method to it one have to substitute the energy of interaction by the self-intersection local time. So the question of the dependence the self-intersection local time on the size of Brownian path is natural. We present the corresponding estimations in the section 4.
 %fin vst
   Some necessary facts from geometry of Hilbert-valued functions are proved in the appendix of the paper.

\section{Existence of local time for Gaussian integrators}
\label{section2}

 Let us recall approaches of defining the local time for the one-dimensional Wiener process $w(t),\ t\in[0; 1].$ Put
 $$
 f_\ve(y)=\frac{1}{\sqrt{2\pi\ve}}e^{-\frac{y^2}{2\ve}},\ y\in\mbR,\ \ve>0.
 $$
 \begin{defn}
 \label{defn3}
 For any $t\in[0; 1]$ and $u\in\mbR$
 $$
 \int^t_0\delta_u(w(s))ds:=L_2\mbox{-}\lim_{\ve\to0}\int^t_0f_\ve(w(s)-u)ds
 $$
 is said to be a local time of the Wiener process at the point $u$ up to time $t.$
 \end{defn}

 Consider an occupation measure of $w$ up to time $t$ defined by the formula
 $$
 \mu_t(D)=\int^t_0\1_D(w(s))ds,\ D\in B(\mbR)
 $$
($B(\mbR)$ is the Borel $\sigma$-field on $\mbR$). $\mu_t(D)$ equals the Lebesgue measure of the time which the trajectory of the Wiener process spends in the set $D$ up to time $t.$ Levy in \cite{5} proved that for almost all trajectories of $w$ and any $t\in[0; 1]$ the measure $\mu_t$ has a density, i.e. there exists the random function $l(u, t),\ u\in\mbR$ such that a.s. for any $t\in[0; 1]$ and $D\in B(\mbR)$
 $$
 \mu_t(D)=\int_Dl(u, t)du.
 $$
 Trotter in \cite{6} proved that the density of an occupation measure of the Wiener process is continuous in $u$ and $t.$ Useful consequence of joint continuity is the following occupation density formula. For every continuous function $\vf$ on $\mbR$ with compact support
 \begin{equation}
 \label{eq4}
 \int^t_0\vf(w(s))ds=\int_{\mbR}\vf(u)l(u, t)du.
 \end{equation}
 It follows from \eqref{eq4} that
 $$
 \int^t_0\delta_u(w(s))ds=\lim_{\ve\to0}\int^t_0f_\ve(w(s)-u)ds=
 $$
 $$
 =\lim_{\ve\to0}\int_{\mbR}f_\ve(v-u)l(v, t)dv=l(u, t).
 $$
 Therefore, the value of the density of occupation measure $l(u, t)$ is the local time of the Wiener process at $u$ up to time $t.$ In this section we will establish the same properties of the local time for Gaussian integrators.

 To prove the existence of the local time for the Gaussian integrator $x$ with the representation \eqref{eq2} we need the notion of local nondeterminism for Gaussian processes introduced by Berman in \cite{7}. Let $\{y(t),\ t\in J\}$ be $\mbR$-valued zero mean Gaussian process on an open interval $J.$ Suppose that there exists $d>0$ such that

 1) $E(y(t)-y(s))^2>0,$ for all
 $
 s, t\in J: 0\leq|t-s|\leq d;
 $

 2) $Ey^2(t)>0$ for all $t\in J.$

 For $m\geq2, \ t_1, \ldots, t_m\in J, \ t_1<t_2<\ldots< t_m $ put
 $$
 V_m=\frac{\Var(y(t_m)-y(t_{m-1})/y(t_1), \ldots, y(t_{m-1}))}
{\Var(y(t_m)-y(t_{m-1}))}
$$
which is the ratio of the conditional and the unconditional variance.

 \begin{defn} [\cite{7}]
 \label{defn4}
 A Gaussian process $y$ is said to be locally nondetermined on $J$ if for every $m\geq2$
 $$
 \lim_{c\to0}\inf_{t_m-t_1\leq c}V_m>0.
 $$
The following statement was proved in \cite{7} and demonstrates that the local nondeterminism property can be used as one of the sufficient conditions for the existence and smoothness of the local time for general Gaussian process.
 \end{defn}
 \begin{thm} [\rm \cite{7}]
 \label{thm5}
 Let $y(t),\ t\in[0; T]$ be a centered Gaussian process satisfying the following three conditions

 1) $y(0)=0;$

 2) $y$ is locally nondetermined on $(0; T);$

 3) there exist positive real numbers $\gamma, \delta$ and $a$ continuous even function $b(t)$ such that $b(0)=0,\  b(t)>0,\  t\in(0; T],$
 $$
 \lim_{h\to0}h^{-\gamma}\int^h_0(b(t))^{-1-2\delta}dt=0
 $$
 and $E(y(t)-y(s))^2\geq b^2(t-s)$ for all $s, t\in[0; T].$ Then there exists a version $l(u, t),\ u\in\mbR,\ t\in[0; T]$ of the local time of the process $y$ which is jointly continuous in $(u, t)$ and which satisfies a H\"older condition in $t$ uniformly in $u,$ i.e. for every $\gamma_1<\gamma,$ i.e. there exist positive and finite random variables $\eta$ and $\eta_1$ such that
 $$
 \sup_u|l(u, t+h)-l(u, t)|\leq \eta_1|h|^{\gamma_1}
 $$
 for all $s, t, t+h\in[0; T]$ and all $|h|<\eta.$
 \end{thm}

 To discuss the existence of the local time for Gaussian integrator $x$ we need reformulation of the notion of local nondeterminism. Denote by $G(e_1, \ldots, e_n)$ the Gram determinant constructed by vectors $e_1, \ldots, e_n.$ Let $g\in C([0; 1], L_2([0; 1])),$\
 $
 \Delta g(t_i)=g(t_{i+1})-g(t_i),\ i=\ov{1, m-1}
 $
 ($C([0; 1], L_2(0; 1))$ is the space of all continuous functions from $[0; 1]$ into $L_2([0; 1])$).

 \begin{lem}
 \label{lem6}
 The Gaussian process $y(t)=(g(t), \xi),$ where $\xi$ is a white noise in $L_2([0; 1])$ is locally nondetermined on $J$ iff for every $m\geq2$
 $$
 \lim_{c\to0}\inf_{t_m-t_1\leq c}
 \frac{G(g(t_1), \Delta g(t_1), \ldots, \Delta g(t_{m-1}))}
 {\|g(t_1)\|^2\|\Delta g(t_1)\|^2\ldots\|\Delta g(t_{m-1})\|^2}>0.
 $$
 \end{lem}
 \begin{proof}
 This is a consequence of definition $V_m$ and relation
 $$
 V_2\cdot\ldots\cdot V_m=\frac
 {\det\cov(x(t_i), x(t_j))^m_{ij=1}}
 {\Var x(t_1)\Var(x(t_2)-x(t_1))\cdot\ldots\cdot\Var(x(t_m)-x(t_{m-1}))}=
 $$
 $$
 \frac{G(g(t_1), \Delta g(t_1), \ldots, \Delta g(t_{m-1}))}
 {\|g(t_1)\|^2\|\Delta g(t_1)\|^2\ldots\|\Delta g(t_{m-1})\|^2}.
 $$
 \end{proof}
 By using the lemma \ref{lem6} one can establish the following statement.
 \begin{thm}[\rm \cite{8}]
 \label{thm8}
 Suppose that the operator $A$ in the representation \eqref{eq2} of $x$ is continuously invertible. Then there exists a version $l(u, t),\ u\in\mbR,\  t\in[0; 1]$ of the local time of $x$ which is jointly continuous in $(u, t)$ and which satisfies a H\"older condition in $t$ uniformly in $u,$ i.e. for every $\gamma<\frac{1}{2}$ there exist positive and finite random variables $\eta$ and $\eta_1$ such that
 $$
 \sup_u|l(u, t+h)-l(u, t)|\leq\eta_1|h|^\gamma
 $$
 for all $s, t, t+h\in[0; 1]$ and all $|h|<\eta.$
 \end{thm}

 The theorem \ref{thm8} was proved in the article of the second author \cite{8}. Here we briefly recall the main steps of the proof, which is based on the following key property of the Gram determinant.
 \begin{lem}
 \label{lem8}
 Suppose that $A$ is a continuously invertible operator in the Hilbert space $H.$ Then for all $k\geq1$ there exists a positive constant $c(k)$ which depends on $k$ such that for any $e_1, \ldots, e_k\in H$ the following relation holds
 $$
 G(Ae_1, \ldots, Ae_k)\geq c(k)G(e_1, \ldots, e_k).
 $$
 \end{lem}

 The lemma \ref{lem8} is proved in the appendix (Lemma A.1).
 \begin{proof}[Proof of the theorem \ref{thm8}.]
 To prove the theorem let us check that $x$ satisfies conditions 1)--3) of the theorem \ref{thm5}. It is obvious that $x(0)=0.$ The lemma \ref{lem6} and the lemma \ref{lem8} imply that $x$ is locally nondetermined. Let us check that $x$ satisfies the condition 3) of the theorem \ref{thm5}. Let $b(t)=c\sqrt{t},\  c>0.$ Pick $\delta<\frac{1}{2}$ and then $\gamma$ such that $\gamma<\frac{1}{2}-\delta.$  One can see that
 $$
 \lim_{h\to0}h^{-\gamma}\int^h_0t^{-\frac{1}{2}-\delta}dt=\frac{2}{1-2\delta}\lim_{h\to0}h^{\frac{1}{2}-\delta-\gamma}=0.
 $$
\end{proof}

   \section{On continuous dependence of local times of integrators on generating them operators}
   \label{section4}

   Suppose that $A_n, A$ are continuously invertible operators in $L_2([0; 1])$ which generate Gaussian integrators $x_n,\ x,$ i.e.
   $$
   x_n(t)=(A_n\1_{[0; t]}, \xi), \ x(t)=(A\1_{[0; t]}, \xi), \ t\in[0; 1].
   $$
   We proved in the section 2 that there exist the random variables
   $$
   l_n(u):=l_n(u, 1)=\int^1_0\delta_u(x_n(t))dt,
   $$
    $$
   l(u):=l(u, 1)=\int^1_0\delta_u(x(t))dt, \ u\in\mbR.
   $$
   The following statement shows that if the sequence of operators $A_n$ converges strongly to an operator $A,$ then the sequence of local times of integrators $x_n$ converges in mean square to the local time of $x.$
   \begin{thm}
   \label{thm11}
   Suppose that $A_n, A$ are continuously invertible operators in $L_2([0; 1])$ such that

   1) for any $y\in L_2([0; 1])$
   $$
   \|A_ny-Ay\|\to0, \ n\to\infty;
   $$

   2) $\sup_{n\geq1}\|A^{-1}_n\|<\infty.$

   Then
   $$
   E\int_{\mbR}(l_n(u)-l(u))^2du\to0, \ n\to\infty.
   $$
      \end{thm}
      \begin{proof}
      To prove the theorem it suffices to check that
      $$
      E\int_{\mbR}l^2_n(u)du\to  E\int_{\mbR}l^2(u)du,\ n\to\infty,
      $$
      $$
      E\int_{\mbR}l_n(u)l(u)du\to  E\int_{\mbR}l^2(u)du,\ n\to\infty.
      $$
      It follows from the theorem B.1  that
      $$
       E\int_{\mbR}l_n(u)l_n(u)du=E\int^1_0\int^1_0\delta_0(x_n(t)-x_n(s))dsdt=
       $$
       $$
       =\lim_{\ve\to0}E\int^1_0\int^1_0f_\ve(x_n(t)-x_n(s))dsdt=
       $$
       $$
       =\frac{2}{\sqrt{2\pi}}\int_{\Delta_2}
       \frac{dsdt}{\|A_n\1_{[s; t]}\|}.
       $$
       It follows from the invertibility of operators $A_n$ and the condition 2) that
       $$
       \frac{1}{\|A_n\1_{[s; t]}\|}\leq\sup_{n\geq1}\|A^{-1}_n\|\frac{1}{\sqrt{t-s}}.
       $$
       The Lebesgue dominated convergence theorem implies that
       $$
       E\int_{\mbR}l^2_n(u)du\to E\int_{\mbR}l^2(u)du, \ n\to\infty.
       $$
       Let us check that
       $$
       E\int_{\mbR}l_n(u)l(u)du\to E\int_{\mbR}l^2(u)du, \ n\to\infty.
       $$
       Again by using the theorem B.1  one can write
       $$
       E\int_{\mbR}l_n(u)l(u)du= E\int_0^1\int^1_0\delta_0(x_n(t)-x(s))dsdt=
       $$
       $$
       =\frac{2}{\sqrt{2\pi}}\int_{\Delta_2}\frac{dsdt}{\|A_n\1_{[0; t]}-A\1_{[0; s]}\|}=
       $$
       $$
       =\frac{2}{\sqrt{2\pi}}\int_{\Delta_2}\frac{dsdt}{\|A_n(\1_{[0; t]}-A^{-1}_nA\1_{[0; s]})\|}\leq
       $$
       $$
       \leq\frac{2}{\sqrt{2\pi}}\sup_{n\geq1}\|A^{-1}_n\|
       \int_{\Delta_2}\frac{dsdt}{\|\1_{[0; t]}-\vk_n(s)\|},
       $$
       where $\vk_n(s)=A^{-1}_nA\1_{[0; s]}.$ It follows from the lemma A.4 (see appendix A) that the sequence $\Big\{\frac{1}{\|\1_{[0 ; t]}-\vk_n(s)\|}\Big\}_{n\geq1}$ is uniformly integrable. Consequently,
       $$
       E\int_{\mbR}l_n(u)l(u)du\to E\int_{\mbR}l^2(u)du, \ n\to\infty.
       $$
\end{proof}

 \section{Conditional moments of Brownian self-intersection local time}
 \label{section2'}
 In this part of the article we will discuss the relationships between the norm of the end-point of Brownian path and its self-intersection local time. As it was mentioned in the introduction, such relation reflects the fact that the real polymers have biggest Flory number then ideal due to the excluded volume effect. Here we will study the conditional distribution of the self-intersection local time for Brownian motion under condition that its end-point tends to infinity. Begin with the one-dimensional Brownian motion $w.$ As it was disscused for example in \cite{25}, the self-intersection local time for $w$ exists. Denote it by
$$
T_2=\int_{\Delta_2}\delta_0(w(t_2)-w(t_1))dt_1dt_2,\ \Delta_2=\{0\leq t_1\leq t_2\leq 1\}.
$$
The following statement holds.

\noindent
{\bf Theorem 10}. {\it For any $p>0$ and $\beta\in(0; 1)$
$$
E(T^p_2/w(1)=a)=O(|a|^{-\beta}), \ a\to\infty.
$$
}

\begin{proof}
It is enough to consider $p$ integer. Then
$$
E(T^p_2/w(1)=a)=
$$
$$
=E\int_{\Delta^p_2}\prod^p_{i=1}\delta_0(\eta(t^i_2)-\eta(t^i_1))d\vec{t},
$$
where $\eta(t)=w(t)-tw(1)+at, \ t\in[0; 1].$ In terms of white noise $\xi=\dot{w}$ the process $\eta$ has a representation
$$
\eta(t)=(Qg_0(t), \xi)+at.
$$
Here $g_0(t)=\1_{[0; t]}$ and $Q$ is a projection onto orthogonal complement to $g_0(1)=\1_{[0; 1]}.$ To estimate the conditional expectation for $T^p_2$ let us use the following lemma from Appendix (Lemma A.5).

\noindent
{\bf Lemma 11}. {\it Let the elements $e_1, \ldots, e_n$ of $L_2([0; 1])$ and a projection $Q$ are such that $Qe_1, \ldots, Qe_n$ are linearly independent. Suppose that the elements $f, g$  satisfy relationships \newline $\forall \ i=1, \ldots, n:$
$$
(f, e_i)=(g, Qe_i).
$$
Then
$$
\|P_1f\|\leq \|P_2g\|,
$$
where $P_1$ and $P_2$ are the orthogonal projections on the linear span of $e_1, \ldots, e_n$ and $Qe_1, \ldots, Qe_n$ respectively.
}

To apply this lemma for our situation denote by $\Gamma^Q_{\vec{t}}$ and $P^Q_{\vec{t}}$ the Gram determinant for $Qe_1, \ldots, Qe_p$ and the projection on its linear span, where $e_i=\1_{[t^i_1; t^i_2]}, i=1, \ldots, p,$ and $Q$ is a projection onto $\1^\perp_{[0; 1]}.$ Then
$$
E\int_{\Delta^p_2}\prod^p_{i=1}\delta_0(\eta(t^i_2)-\eta(t^i_1))d\vec{t}=
$$
$$
=\int_{\Delta^p_2}
\frac{
e^{
-\frac{1}{2}\|P^Q_{\vec{t}}h_{\vec{t}}\|^2a^2
}
}
{\Gamma^Q_{\vec{t}}}d\vec{t}.
$$
Here $h_{\vec{t}}$ is taken in a such way that \newline $\forall \ i=1, \ldots, p:$
$$
(h_{\vec{t}}, Qe_i)=t^i_2-t^i_1.
$$
It follows from the previous lemma, that
$$
\|P^Q_{\vec{t}}h_{\vec{t}}\|\geq \|P_{\vec{t}}\1_{[0; 1]}\|,
$$
where $P_{\vec{t}}$ is a projection onto the linear span of $\1_{[t^1_1; t^1_2]}, \ldots, \1_{[t^p_1; t^p_2]}.$ Consequently, for arbitrary $k=1, \ldots, p$
$$
e^{-\frac{1}{2}\|P^Q_{\vec{t}}h_{\vec{t}}\|^2a^2}\leq e^{-\frac{1}{2}(t^k_2-t^k_1)a^2}.
$$
To find the estimation for $\Gamma^Q_{\vec{t}}$ note that
$$
\Gamma^Q_{\vec{t}}=\Gamma(\1_{[0; 1]}, \1_{[t^1_1; t^1_2]}, \ldots, \1_{[t^p_1; t^p_2]}).
$$
Let us use the following lemma (Lemma A.6).

\noindent
{\bf Lemma 12}. {\it
Let $\Delta_0=\O,$ and $\Delta_1, \ldots, \Delta_n$ be the subsets of $[0; 1].$ Then
$$
\Gamma(\1_{\Delta_1}, \ldots, \1_{\Delta_n})\geq\prod^n_{k=1}|\Delta_k\setminus\mathop{\cup}\limits^{k-1}_{j=1}\Delta_j|.
$$
}

As a consequence of this lemma one can obtain the following estimation for the Gram determinant

$$
\Gamma^Q_{\vec{t}}\geq\prod^N_{j=1}|\wt{\Delta}_j|,
$$
where $\wt{\Delta}_j, j=1, \ldots, N$ are the intervals from the partition of $[0; 1]$ by the end-points of the intervals $[t^1_k, t^2_k], k=1, \ldots, p.$
Now using the previous estimation for $\|P^Q_{\vec{t}}h_{\vec{t}}\|$ one can get that
$$
E\int_{\Delta^p_2}\prod^p_{i=1}\delta_0(\eta(t^i_2)-\eta(t^i_1))d\vec{t}\leq
$$
$$
\leq(2p)!\frac{1}{\sqrt{2\pi}^p}\int_{\Delta_{2p}}
\frac
{e^{-\frac{1}{2}a^2(t_{2p}-t_{2p-1})}}
{(\prod^{2p}_{j=0}(t_{j+1}-t_j))^{\frac{1}{2}}}d\vec{t},
$$
where $t_0=1$ and $t_{2p+1}=1.$ Consider the integral with respect to the last variable $t_{2p}$
$$
\int^1_{2p-1}\frac
{e^{-\frac{1}{2}a^2(t_{2p}-t_{2p-1})}}
{\sqrt{(t_{2p}-t_{2p-1})(1-t_{2p})}}dt_{2p}.
$$
Using expression $\delta=1-t_{2p-1}$ and changing variable one can rewrite the last integral as
%vst
$$
\int^\delta_0\frac{e^{-\frac{1}{2}a^2s}}
{\sqrt{s(\delta-s)}}ds=
\int^1_0\frac{e^{-\frac{1}{2}a^2\delta s'}}
{\sqrt{s'(1-s')}}ds'.
$$
Using H\"{o}lder inequality one can get for $\alpha\in(1; 2)$
$$
\int^1_0\frac{e^{-\frac{1}{2}a^2\delta s'}}
{\sqrt{s'(1-s')}}ds'\leq
c_\alpha\Big(
\int^1_0e^{-\frac{1}{2}a^2\delta\frac{\alpha}{\alpha-1} s'ds'}
\Big)^{\frac{\alpha}{\alpha-1}}\leq
$$
$$
\leq\wt{c}_\alpha
\frac{1}{a^{2-\frac{2}{\alpha}}}
\frac{1}{\delta^{1-\frac{1}{\alpha}}},
$$
where $c_\alpha$ and $\wt{c}_\alpha$ are the positive constants which depend on $\alpha.$ Finally, for any $\alpha\in(1; 2)$
$$
E\int_{\Delta^p_2}\prod^p_{i=1}\delta_0(\eta(t^i_2)-\eta(t^i_1))d\vec{t}\leq
$$
$$
\leq\wt{\wt{c}}_\alpha
\int_{\Delta_{2p-1}}
\prod^{2p-2}_{j=0}
\frac{1}{\sqrt{t_{j+1}-t_j}}\cdot
\frac{1}{(1-t_{2p-1})^{1-\frac{1}{\alpha}+\frac{1}{2}}}d\vec{t}
\cdot a^{-2+\frac{2}{\alpha}}.
$$
\end{proof}
For planar Wiener process $w$ on the interval $[0; 1]$ consider trajectories with $w(1)=a.$ One can expect that if $\|a\|$ is large, then the trajectory of $w$ has a small number of self-intersections. The conditional distribution of the Wiener process under the condition $w(1)=a$ coincides with the distribution of the Brownian bridge
$$
y_a(t)=w(t)-tw(1)+at, \ t\in[0; 1].
$$
Let us investigate the dependence of the self-intersection local time of the process $y_a(t),\ t\in[0; 1]$ on $\|a\|.$ Denote by
$$
T_2(a, \alpha)=\int_{\Delta_2(a, \alpha)}
\delta_0(y_a(t_2)-y_a(t_1))dt_1dt_2,
$$
where
$$
\Delta_2(a, \alpha)=\{(t_1, t_2): \ 0\leq t_1\leq 1-\|a\|^{-\alpha},\ t_1+\|a\|^{-\alpha}\leq t_2\leq1\}.
$$
The self-intersection local time
$$
\int_{\Delta_2(a,\alpha)}\delta_{0}(w(t_2)-w(t_1))dt_1dt_2
$$
exists (see \cite{25}). As before, one can check that
$$
T_{2}(a,\alpha)=E\left(\int_{\Delta_2(a,\alpha)}\delta_{0}(w(t_2)-w(t_1))dt_1dt_2/w(1)=a\right).
$$
 Let us prove that the following statement holds.

\noindent
{\bf Theorem 13}. {\it
For

1) $\alpha=2:$
$$
\lim_{\|a\|\to+\infty}ET_2(a, \alpha)=\frac{1}{2\pi}\int^{+\infty}_1\frac{1}{y}e^{-\frac{y}{2}}dy;
$$

2) $\alpha<2:$
$$
\lim_{\|a\|\to+\infty}ET_2(a, \alpha)=0;
$$

3) $\alpha>2:$
$$
\lim_{\|a\|\to\infty}ET_2(a, \alpha)=+\infty.
$$
}

\begin{proof}
Let $\Delta t_1=t_2-t_1.$
\begin{equation}
\label{eq1_9}
ET_2(a, \alpha)=
\int_{\Delta_2(a, \alpha)}
\frac{1}{2\pi\Delta t_1(1-\Delta t_1)}e^{\frac{1}{2}\frac{\Delta t_1\|a\|^2}{1-\Delta t_1}}d\vec{t}.
\end{equation}
Let $t_1=s_1, \ \|a\|^2\Delta t_1=s_2.$ Then \eqref{eq1_9} equals
$$
\int^{1-\|a\|^{-\alpha}}_0\int^{(1-s_1)\|a\|^2}_{\|a\|^{-\alpha+2}}
\frac{\|a\|^2}{2\pi s_2(\|a\|^2-s_2)}
e^{-\frac{1}{2}\frac{\|a\|^2s_2}{\|a\|^2-s_2}}ds_2ds_1=
$$
$$
=\int^{\|a\|^2}_{\|a\|^{-\alpha+2}}
\int^{1-\frac{s_2}{\|a\|^2}}_{0}
\frac{\|a\|^2}{2\pi s_2(\|a\|^2-s_2)}
e^{-\frac{1}{2}\frac{\|a\|^2s_2}{\|a\|^2-s_2}}ds_1ds_2=
$$
$$
=
\int^{\|a\|^2}_{\|a\|^{-\alpha+2}}
\Big(1-\frac{s_2}{\|a\|^2}\Big)
\frac{\|a\|^2}{2\pi s_2(\|a\|^2-s_2)}
e^{-\frac{1}{2}\frac{\|a\|^2s_2}{\|a\|^2-s_2}}ds_2=
$$
\begin{equation}
\label{eq2_9}
=
\frac{1}{2\pi}
\int^{\|a\|^2}_{\|a\|^{-\alpha+2}}
\frac{1}{s_2}
e^{-\frac{1}{2}\frac{\|a\|^2s_2}{\|a\|^2-s_2}}ds_2.
\end{equation}
Put $\frac{\|a\|^2s_2}{\|a\|^2-s_2}=y.$ Then \eqref{eq2_9} equals
\begin{equation}
\label{eq3_9}
\frac{1}{2\pi}
\int^{+\infty}_{\frac{\|a\|^{2-\alpha}}{1-\|a\|^{-\alpha}}}
\frac{\|a\|^2}{y(\|a\|^2+y)}e^{-\frac{y}{2}}dy.
\end{equation}
Note that for $\alpha=2$
$$
\frac{1}{2\pi}\int^{+\infty}_{\frac{1}{1-\|a\|^{-2}}}
\frac{\|a\|^2}{y(\|a\|^2+y)}e^{-\frac{y}{2}}dy\to
\frac{1}{2\pi}\int^{+\infty}_1\frac{1}{y}e^{-\frac{y}{2}}dy, \ \|a\|\to+\infty.
$$
One can see that for $\alpha<2$
$$
\frac{1}{2\pi}\int^{+\infty}_{\frac{\|a\|^{2-\alpha}}{1-\|a\|^{-2}}}
\frac{\|a\|^2}{y(\|a\|^2+y)}e^{-\frac{y}{2}}dy\leq
$$
\begin{equation}
\label{eq4_9}
\leq\frac{1}{2\pi}
\frac{(1-\|a\|^{-\alpha})^2}{\|a\|^{2-\alpha}}
e^{-\frac{\|a\|^{2-\alpha}}{1-\|a\|^{-\alpha}}}.
\end{equation}
The estimate \eqref{eq4_9} implies that for $\alpha<2$
$$
\lim_{\|a\|\to+\infty}
\frac{1}{2\pi}
\int^{+\infty}_{\frac{\|a\|^{2-\alpha}}{1-\|a\|^{-\alpha}}}
\frac{\|a\|^2}{y(\|a\|^2+y)}
e^{-\frac{y}{2}}dy=0.
$$
On the other hand for $\alpha>2$
$$
\frac{1}{2\pi}
\int^{+\infty}_{\frac{\|a\|^{2-\alpha}}{1-\|a\|^{-\alpha}}}
\frac{\|a\|^2}{y(\|a\|^2+y)}
e^{-\frac{y}{2}}dy\geq
$$
$$
\geq
\frac{1}{2\pi}
\int^{\frac{1}{m}}_{\frac{\|a\|^{2-\alpha}}{1-\|a\|^{-\alpha}}}
\frac{\|a\|^2}{y(\|a\|^2+y)}
e^{-y}dy\geq
$$
\begin{equation}
\label{eq5_9}
\geq
\frac{1}{2\pi}
\frac{m\|a\|^2}{\|a\|^2+\frac{1}{m}}
\Big(-e^{\frac{1}{m}}+e^{-\frac{\|a\|^{2-\alpha}}{1-\|a\|^{-\alpha}}}\Big).
\end{equation}

It follows from \eqref{eq5_9} that for $\alpha>2$
$$
\lim_{\|a\|\to+\infty}
\frac{1}{2\pi}
\int^{+\infty}_{\frac{\|a\|^{2-\alpha}}{1-\|a\|^{-\alpha}}}
\frac{\|a\|^2}{y(\|a\|^2+y)}
e^{-\frac{y}{2}}dy=+\infty.
$$

\end{proof}

    \section*   { Appendix A. {\large On some geometry of Hilbert-valued functions}}

In this appendix we collect some useful estimates for Gramian matrix and Gram determinant which describe the changing of geometry of Hilbert-valued functions under the action of a linear continuous operator. Also for $ 0\leq\alpha<1$ we check that
    $$
    \sup_{y\in L_2([0;1])}\int^1_0\frac{dt}{\|\1_{[0; t]}-y\|^{1+\alpha}}.
    $$

Let $B(e_1, \ldots, e_n)$ be the Gramian matrix constructed from the vectors $e_1,\ldots, e_n$ in the Hilbert space $H.$

    \noindent
    {\bf Lemma A.1.} {\it Suppose that $A$ is a continuously invertible operator in the Hilbert space $H.$ Then for all $k\geq1$ there exists a positive constant $c(k)$ which depends on $k$ and $A$ such that for any $e_1, \ldots, e_k\in H$ the following relation holds
    $$
    G(Ae_1, \ldots, Ae_k)\geq c(k)G(e_1, \ldots, e_k).
    $$
}

    \begin{proof}
    To prove the lemma it suffices to check that
    $$
    \inf G\left(\frac{Af_1}{\|Af_1\|}, \ldots,\frac{Af_k}{\|Af_k\|} \right)>0,
    $$
    where infimum is taking over all orthonormal systems $(f_1, \ldots, f_k).$ Using the Gram--Schmidt orthogonalization procedure build the orthogonal system $q_1, \ldots, q_k$ from $\frac{Af_1}{\|Af_1\|}, \ldots,\frac{Af_k}{\|Af_k\|}.$

    Here
    $$
    q_j=\frac{Af_j}{\|Af_j\|}-\sum^{j-1}_{i=1}a_{ij}\frac{Af_i}{\|Af_i\|}
    $$
    with some $a_{ij}.$ Let us prove that
     $$
   \inf_{(f_1, \ldots, f_k)} G\left(\frac{Af_1}{\|Af_1\|}, \ldots,\frac{Af_k}{\|Af_k\|} \right)=
    $$
    $$
    =\inf_{(f_1, \ldots, f_k)}\prod^k_{i=1}\|q_i\|^2>0.
    $$
    If it is not so, then there exists the sequence $\{f^n_1, \ldots, f^n_k\}_{n\geq1}$ and $j=\ov{1, k}$ such that $\|q^n_j\|\to0,\ n\to\infty.$ The invertibility of the operator $A$ implies that
    $$
    \left\|\frac{f^n_j}{\|Af^n)j\|}-\sum^{j-1}_{i=1}a^n_{ij}\frac{f^n_i}{\|Af^n_i\|}
\right\|\to0, n\to\infty.
$$
But
$$
 \left\|\frac{f^n_j}{\|Af^n_j\|}-\sum^{j-1}_{i=1}a^n_{ij}\frac{f^n_i}{\|Af^n_i\|}
\right\|\geq\frac{1}{\|Af^n_j\|}>0.
$$
\end{proof}

       \noindent
       {\bf Lemma A.2.}
{\it
Suppose that $A$ is a continuously invertible operator in the Hilbert space $H.$ Then for any $e_0=0,\ e_1, \ldots, e_n \in H$ such that $e_{i+1}-e_i\perp e_{j+1}-e_j,\ i,j=\ov{1, n-1},\  i\ne j$ there exists a positive constant $c$ such that for all $\vec{u}\in\mbR^n$ with $u_0=0$ the following relation holds
$$
(B^{-1}(Ae_1, \ldots, Ae_n)\vec{u}, \vec{u})\geq c\sum^{n-1}_{i=0}
\frac{(u_{i+1}-u_i)^2}{\|e_{i+1}-e_i\|^2}.
$$
}
\begin{proof}
It was proved in \cite{10} that in the case $\vec{u}=((h_0, Ae_1), \ldots, (h_0, Ae_n)), $ \ $h_0\in H$ the following relation holds
$$
(B^{-1}(Ae_1, \ldots, Ae_n)\vec{u}, \vec{u})=\|P_{Ae_1\ldots Ae_n}h_0\|^2,
$$
where $P_{e_1\ldots e_n}$ is a projection on $LS\{e_1, \ldots, e_n\}$ (linear span generated by elements $e_1, \ldots, e_n$). Note that
$$
((h_0, Ae_1), \ldots, (h_0, Ae_n))=((A^*h_0, e_1), \ldots, (A^*h_0, e_n)).
$$
Since $(A^*h_0, e_1)=u_1, \ldots, (A^*h_0, e_n)=u_n,$ then
$$
A^*h_0=\sum^{n-1}_{i=0}\frac{e_{i+1}-e_i}{\|e_{i+1}-e_i\|^2}(u_{i+1}-u_i)+r,
\eqno(A.1)
$$
where $r\perp e_i,\ i=\ov{0, n}.$

Consequently,
$$
h_0=\sum^{n-1}_{i=0}{A^*}^{-1}\Big(\frac{e_{i+1}-e_i}{\|e_{i+1}-e_i\|^2}(u_{i+1}-u_i)\Big)+{A^*}^{-1}r.
$$
Let us remark that  continuous invertibility of the operator $A$ implies the existence of ${A^*}^{-1}.$ It follows from (A.1) that
$$
(B^{-1}(Ae_1, \ldots, Ae_n)\vec{u}, \vec{u})=
$$
$$
=\Big\|{A^*}^{-1}\Big(\sum^{n-1}_{i=0}
\frac{(e_{i+1}-e_i)(u_{i+1}-u_i)}{\|e_{i+1}-e_i\|^2}+r\Big)\Big\|^2
$$
$$
\geq c\sum^{n-1}_{i=0}
\frac{(u_{i+1}-u_i)^2}{\|e_{i+1}-e_i\|^2}+c\|r\|^2\geq
$$
$$
\geq
c\sum^{n-1}_{i=0}
\frac{(u_{i+1}-u_i)^2}{\|e_{i+1}-e_i\|^2}.
$$

\end{proof}

The following statement describes the direct application of lemma A.1 and lemma A.2.

For $s_1, \ldots, s_n\in \Delta_n, u_1, \ldots, u_n\in\mbR$ let $p_{s_1\ldots s_n}(u_1, \ldots, u_n)$ be the density of Gaussian vector $(x(s_1), \ldots, x(s_n))$ in $\mbR^n.$ Here $x$ is Gaussian integrator with the representation \eqref{eq2}.

\noindent
{\bf Lemma A.3} \cite{8}. {\it Suppose that $A$ in the representation \eqref{eq2} is continuously invertible. Then there exist positive constants $c_1(n), c_2$ such that the following relation holds
$$
p_{s_1\ldots s_n}(u_1, \ldots, u_n)\leq
\frac{c_1(n)}
{\sqrt{s_1(s_2-s_1)\ldots(s_n-s_{n-1})}}e^{-\frac{1}{2}c_2}\sum^{n-1}_{j=0}\frac{(u_{j+1}-u_j)^2}{s_{j+1}-s_j}.
$$
}
For a proof see \cite{8}.

\noindent
{\bf Lemma A.4.}
{\it
For any $0\leq\alpha<1$
$$
\sup_{y\in L_2([0;1])}\int^1_0\frac{1}{\|\1_{[0; t]}-y\|^{1+\alpha}}dt<+\infty.
$$
}

\begin{proof}
Put $g_0(t):=\1_{[0; t]}.$ Note that
$$
\int^1_0\frac{1}{\|g_0(t)-y\|^{1+\alpha}}dt=\int^{+\infty}_0\lambda\{t: \|g_0(t)-y\|^{-(1+\alpha)}\geq z\}dz,
$$
where $\lambda$ is the Lebesgue measure on $
\mbR.$ Then to prove the statement of the lemma it suffices to check that for $b>0$
$$
\sup_{y\in L_2([0;1])}\int^{+\infty}_b\lambda\{t: \|g_0(t)-y\|^{-(1+\alpha)}\geq z\}dz<+\infty.
$$
For any $g_0(t_0), g_0(t_1)$ from the closed ball $\ov{B}\Big(y, \frac{1}{z^{\frac{1}{1+\alpha}}}\Big)$ the following relation holds
$$
|t_0-t_1|=\|g_0(t_0)-g_0(t_1)\|^2{\leq}\frac{4}{z^{\frac{2}{1+\alpha}}}.\eqno(A.2)
$$
(A.2) implies that
$$
\Big\{t: \|g_0(t)-y\|\leq\frac{1}{z^{\frac{1}{1+\alpha}}}\Big\}\subset\Big[t_0-\frac{4}{z^{\frac{2}{1+\alpha}}}; t_0+\frac{4}{z^{\frac{2}{1+\alpha}}}\Big],
$$
for some $t_0$ such that
$$
\|g_0(t)-y\|\leq\frac{1}{z^{\frac{1}{1+\alpha}}}.\eqno(A.3)
$$
It follows from (A.3) that for $0\leq\alpha<1$
$$
\int^\infty_b
\lambda\{t:
\|g_0(t)-y\|\leq\frac{1}{z^{\frac{1}{1+\alpha}}}
\}dz\leq 4\int^{+\infty}_b\frac{dz}{z^{\frac{2}{1+\alpha}}}<+\infty.
$$
\end{proof}

\noindent
{\bf Lemma A.5}. {\it Let the elements $e_1, \ldots, e_n$ of $L_2([0; 1])$ and a projection $Q$ are such that $Qe_1, \ldots, Qe_n$ are linearly independent. Suppose that the elements $f, g$  satisfy relationships \newline $\forall \ i=1, \ldots, n:$
$$
(f, e_i)=(g, Qe_i).
$$
Then
$$
\|P_1f\|\leq \|P_2g\|,
$$
where $P_1$ and $P_2$ are the orthogonal projections on the linear span of $e_1, \ldots, e_n$ and $Qe_1, \ldots, Qe_n$ respectively.
}

\begin{proof}
Note, that \newline $\forall \ i=1, \ldots, n$
$$
(g, Qe_i)=(P_2g, Qe_i)=(QP_2g, e_i)=(P_1f, e_i).
$$
Consequently,
$$
P_1f=P_1QP_2g=P_1P_2g.
$$
Hence
$$
\|P_1f\|\leq\|P_2g\|.
$$
\end{proof}

\noindent
{\bf Lemma A.6}. {\it
Let $\Delta_0=\O,$ and $\Delta_1, \ldots, \Delta_n$ be the subsets of $[0; 1].$ Then
$$
\Gamma(\1_{\Delta_1}, \ldots, \1_{\Delta_n})\geq\prod^n_{k=1}|\Delta_k\setminus\mathop{\cup}\limits^{k-1}_{j=1}\Delta_j|.
$$
}
\begin{proof}
Since
$$
\Gamma(\1_{\Delta_1}, \ldots, \1_{\Delta_n})=|\Delta_1|\prod^n_{k=2}\|h_k\|^2,
$$
where $h_k$ is the orthogonal component of $\1_{\Delta_k}$ with respect to the linear span of $\1_{\Delta_1}, \ldots, \1_{\Delta_{k-1}},$ then it is enough to prove that for $k=2, \ldots, n$
$$
\|h_k\|^2\geq|\Delta_k\setminus\mathop{\cup}\limits^{k-1}_{j=1}\Delta_j|.
$$
For the set $\Delta$ such that $|\Delta|>0$ denote by $P_\Delta$ the orthogonal projection into $\1_\Delta.$ Then
$$
\|P_{\Delta_j}\1_{\Delta_k}\|^2=\frac{|\Delta_k\cap\Delta|^2}{|\Delta|}\leq\big|\Delta_k\cap\Delta\big|.
$$
Consider the representation
$$
\cup^{k-1}_{j=1}\Delta_j=\cup^l_{i=1}H_i,
$$
where $|H_i|>0$ and disjoint, all $H_i$ belong to the algebra generated by $\{\Delta_j\}$ and every $\Delta_j$ can be obtained as the union of the certain $H_i.$ Then the linear span of $\1_{\Delta_1, \ldots, \1_{\Delta_{k-1}}}$ is a subset of the linear span of $\1_{H_1}, \ldots, \1_{H_l}.$ Hence
$$
\|h_k\|^2\geq|\Delta_k|-\sum^l_{i=1}|\Delta_k\cap H_i|=|\Delta_k\setminus\mathop{\cup}\limits^{k-1}_{j=1}\Delta_j|.
$$
\end{proof}

\section*{Appendix B. {\large On some relations between generalized func\-tio\-nals from white noise}}
Consider linearly independent elements $f_1, \ldots, f_n\in L_2([0; 1]).$ Here we investigate conditions on elements $r_j\in L_2([0; 1]),\ j=\ov{1, n-1}$ that allow to establish the following relation
$$
\int_{\mbR}\prod^n_{k=1}\delta_0((f_k, \xi)-u)du=\prod^{n-1}_{j=1}\delta_{0}((r_j, \xi)),
\eqno{\rm(B.1)}
$$
which is understood as equality of the generalized functionals from white noise \cite{15} and will be checked using Fourier--Wiener transform.

The following statement discribes one of the possible choices for $r_j,\  j=\ov{1, n-1}.$

\noindent
{\bf Theorem B.1.} {\it
Let $f_1, \ldots, f_n$ be linearly independent elements in $L_2([0; 1]).$ Then
$$
\int_{\mbR}\prod^n_{k=1}\delta_0((f_k, \xi)-u)du=\prod^{n-1}_{k=1}\delta_0((f_{k+1}-f_k, \xi)).
\eqno{\rm (B.2)}
$$
}

\begin{proof}
To prove the statement let us calculate the Fourier--Wiener transform of the left-hand side and the right-hand side of the equality (B.2). Denote by $\cT(\alpha)(h)$ the Fourier--Wiener transform of random variable $\alpha.$  One can check that
$$
\cT\Big(\prod^{n-1}_{j=1}\delta_0((r_j, \xi))\Big)(h)=
$$
$$
=\frac{1}{(2\pi)^{\frac{n-1}{2}}\sqrt{G(r_1,\ldots, r_{n-1})}}
e^{-\frac{1}{2}\|P_{r_1\ldots r_{n-1}}h\|^2}
\eqno{\rm (B.3)}
$$
(see \cite{1}). Let us find the Fourier--Wiener transform of  $\int_{\mbR}\prod^n_{k=1}\delta_0((f_k, \xi)-u)du$

$$
\cT\Big(\int_{\mbR}\prod^{n}_{k=1}\delta_0((f_k, \xi)-u)du\Big)(h)=
$$
$$
=\int_{\mbR}\frac{1}{(2\pi)^{\frac{n}{2}}\sqrt{G(f_1,\ldots, f_{n})}}
e^{-\frac{1}{2}(B^{-1}(f_1,\ldots, f_n)(u\vec{e}-\vec{a}), u\vec{e}-\vec{a})}du,
\eqno{\rm(B.4)}
$$
where $\vec{e}=\begin{pmatrix}1\\
\vdots\\
1\end{pmatrix}, $ $\vec{a}=\begin{pmatrix}(f_1, h)\\
\vdots\\
(f_n, h)
\end{pmatrix}.$  By integrating (B.4) over $u$ one can get
$$
\frac{1}
{(2\pi)^{\frac{n-1}{2}}
\sqrt{G(f_1, \ldots, f_n)}\sqrt{(B^{-1}(f_1,\ldots, f_n)e, e)}}\cdot
$$
$$
\cdot \exp\Big\{-\frac{1}{2}\Big((B^{-1}(f_1,\ldots, f_n)a,a)-
\frac{(B^{-1}(f_1,\ldots, f_n)a,e)^2}
{(B^{-1}(f_1,\ldots, f_n)e, e)}\Big)\Big\}.
\eqno{\rm(B.5)}
$$
It is not difficult to check that
$$
(B^{-1}(f_1,\ldots, f_n)a,a)=\|P_{f_1\ldots f_n}h\|^2.
$$
Consider the function $f\in LS\{f_1, \ldots, f_n\}$ such that $(f, f_k)=1,\  k=\ov{1, n}.$
Then
$$
(B^{-1}_{f_1\ldots f_n}\vec{e}, \vec{e})=\|P_{f_1\ldots f_n}f\|^2=\|f\|^2
$$
$$
(B^{-1}_{f_1\ldots f_n}\vec{a}, \vec{e})=(P_{f_1\ldots f_n} h, f).
$$
Therefore, (B.5) equals
$$
\frac{1}
{(2\pi)^{\frac{n-1}{2}}\sqrt{
G(f_1, \ldots, f_n)}\|f\|}
e^{-\frac{1}{2}(\|P_{f_1\ldots f_n}h\|^2-\|P_fP_{f_1\ldots f_n}h\|^2)}.
$$
Denote by $f\overset{\perp}{=}\{v\in LS\{f_1, \ldots, f_n\}: (v, f)=0\}.$ Then
$$
\cT\Big(\int_{\mbR}\prod^{n}_{k=1}\delta_0((f_k, \xi)-u)du\Big)(h)=
$$
$$
=\frac{1}
{(2\pi)^{\frac{n-1}{2}}\sqrt{
G(f_1, \ldots, f_n)}\|f\|}
e^{-\frac{1}{2}\|P_{f^\perp}h\|^2}.
\eqno{\rm(B.6)}
$$
By comparing (B.3) and (B.6) we obtain the following conditions on elements $r_k,\ k=\ov{1, n-1}$

1) $LS\{r_1, \ldots, r_{n-1}\}=f^\perp;$

2) $G(r_1, \ldots, r_{n-1})=G(f_1, \ldots, f_n)\|f\|^2.$

Let us check that $r_j:=f_{j+1}-f_j$ satisfy conditions 1), 2).  Really, put $M=LS\{f_2-f_1, \ldots, f_n-f_{n-1}\}.$ Then $f\perp M.$ Denote by $r$ the distance from $f_1$ to $M.$ One can see that
$$
G(f_1,\ldots, f_n)=G(f_1, f_2-f_1, \ldots, f_n-f_{n-1})=r^2G(f_2-f_1, \ldots, f_n-f_{n-1}).
$$
Since
$$
(f_1, \frac{f}{\|f\|})=\|f_1\|\cos\alpha=r,
$$
then $r=\frac{1}{\|f\|}.$ Consequently,
$$
\|f\|^2G(f_1, \ldots, f_{n-1})=G(f_2-f_1, \ldots, f_n-f_{n-1}).
$$
\end{proof}


\begin{thebibliography}{99}
\bibitem{1}
A.A. Dorogovtsev, Stochastic integration and one class of Gaussian random processes, Ukr. Math. journal 50 (4) (1998) 495-505.
\bibitem{2}
A.A. Dorogovtsev, Smoothing prolem in anticipating scenario, Ukr.  Math. Journal 57 (9) (2005)  1424-1441.
\bibitem{3}
A.V. Skorokhod, On some generalization of stochastic integral, Theory of Probability and its Applications XX (2) (1975) 223-238.
\bibitem{4}
A.A. Dorogovtsev, Stochastic Analysis and Random Maps in Hilbert space, VSP, Utrecht, 1994.
\bibitem{5}
P. L\'evy, Sur certains processes stochastiques homog\'enes, Compositio Mathematica 7 (2) (1939) 283-339.
\bibitem{6}
H.F. Trotter, A property of Brownian motion paths, Illinois J.Math. 2 (3) (1958) 425-433.
\bibitem{7} S.M. Berman, Local nondeterminism and local times of Gaussian processes, Indiana University Mathematics Journal 23 (1) (1973) 69-94.
\bibitem{8}
O. Izyumtseva, On the local times for Gaussian integrators, Theory of Stochastic Processes 19 (35) (2014) 11-25.
\bibitem{9} B. Simon, The $P(\vf)_2$ euclidian (quantum) field theory, Princeton University Press, 1974.
\bibitem{10}
A. Dorogovtsev, O. Izyumtseva, Self-intersection local time for Gaussian processes, Lap Lambert Academic Publishing, Germany, 2011.
\bibitem{11}
S. Varadhan, Appendix to Euclidian quantum field theory, by K. Symanzik, R. Jost, New York, 1969.
\bibitem{12}
E.B. Dynkin, Regularized self-intersection local times of planar Brownian motion, Ann. Probab. 16 (1) (1988) 58-74.
\bibitem{13}
J. Rosen, Joint continuity of renormalized intersection local times, Ann. Inst. Henri Poincare 32 (6) (1996) 671-700.
\bibitem{14}
J.-F. Le Gall, Fluctuation results for the Wiener sausage, Ann. Probab. 16 (3) (1988) 991-1018.
\bibitem{15}
S. Watanabe, Stochastic differential equation and Malliavin calculus, Springer-Verlag, Berlin, Heidelberg, 1984.
\bibitem{16}
S. Janson, Gaussian Hilbert spaces, Cambridge university press, 1997.
\bibitem{17}
Y. Xiao, Strong local nondeterminism and sample path properties of Gaussian random fields, Theory in Probability and Statistics with Applications, Higher Education press, Beijing, (2007) 136-176.
\bibitem{18}
J. Kuelbs, W. V. Li, Q.-M. Shao, Small ball probabilities for Gaussian processes with stationary increments under H\"{o}lder norms, J. Theoret. Probab. 8 (1995) 361-386.
\bibitem{19}
S.M. Berman, Spectral conditions for local nondeterminism, Stochastic Process. Appl. 27 (1988) 73-84.
\bibitem{20}
A. Dvoretzky, P. Erd\"{o}s, S. Kakutani, Multiple points of paths of Brownian motion in the plane, Bull. Res. Cousil Israel 3 (1954) 364-371.
\bibitem{21}
J. Rosen, A renormalized local time for multiple intersection of planar Brownian motion, Sem. de Prob. XX 20 (1986) 515-531.
\bibitem{22}
I. Teraoka, Polymer solutions, An introduction to physical properties, Wiley and Sons, New York, 2002.
\bibitem{23}
X. Chen, Limit laws for a energy of charged polymer, Annales de l'Institut Henri Poincare-Probabilites et Statistiques 44 (4) (2008) 638-672.
\bibitem{24}
F. den Hollander, Random polymers, Lecture Notes in Mathematics, Springer-Verlag, Berlin, Heidelberg, 2009.
\bibitem{25}
X. Chen, Random walk intersections: large deviations and some related topics, Mathematical Surveys and Monographs 157, American Mathematical Society, 2010.
\end{thebibliography}
  \end{document}